\documentclass{daj}

\usepackage{amsmath,amscd,amsfonts,amssymb,amsthm}
\usepackage{mathrsfs,dsfont}
\usepackage{color}
\usepackage{mathtools}

\def\R{\mathbb{R}}

\def\Q{\mathbb{Q}}
\def\Z{\mathbb{Z}}
\def\N{\mathbb{N}}

\def\F{\mathcal{F}}

\def\one{\mathds{1}}
\def\eps{\varepsilon}

\renewcommand\leq{\leqslant}
\renewcommand\geq{\geqslant}

\renewcommand\hat{\widehat}

\newcommand{\ft}[1]{\widehat #1}

\newtheorem{thm}{Theorem}[section]
\newtheorem{lem}[thm]{Lemma}
\newtheorem{corollary}[thm]{Corollary}
\newtheorem{prop}[thm]{Proposition}

\newtheorem{question}[thm]{Question}
\newtheorem{conj}[thm]{Conjecture}

\newcommand{\thmref}[1]{Theorem~\ref{#1}}
\newcommand{\secref}[1]{Section~\ref{#1}}
\newcommand{\subsecref}[1]{Subsection~\ref{#1}}
\newcommand{\lemref}[1]{Lemma~\ref{#1}}

\newcommand{\propref}[1]{Proposition~\ref{#1}}

\newcommand{\quesref}[1]{Question~\ref{#1}}

\newtheorem{remark}[thm]{Remark}

\newenvironment{enumerate-math}
{\begin{enumerate}
		\addtolength{\itemsep}{5pt}
		}
	{\end{enumerate}}

\newenvironment{enumerate-text}
{\begin{enumerate}
		\addtolength{\itemsep}{5pt}
		}
	{\end{enumerate}}

\dajAUTHORdetails{%
  title = {The Structure of Translational Tilings in $\mathbb{Z}^d$}, 
  author = {Rachel Greenfeld and Terence Tao},
  plaintextauthor = {Rachel Greenfeld and Terence Tao},
  plaintexttitle = {The Structure of Translational Tilings in Z^d}, 
  runningtitle = {Structure of Translational Tilings in $\mathbb{Z}^d$}, 
  keywords = {tiling, periodicity, decidability},
}   

\dajEDITORdetails{%
   year={2021},
   number={16},
   received={4 November 2020},   
   published={24 September 2021},  
   doi={10.19086/da.28324},       
}   

\begin{document}

\begin{frontmatter}[classification=text]

\title{The Structure of Translational Tilings in $\mathbb{Z}^d$} 

\author[RG]{Rachel Greenfeld\thanks{Partially supported by the Eric and Wendy Schmidt Postdoctoral Award.}}
\author[tao]{Terence Tao\thanks{Partially supported by NSF grant DMS-1764034 and by a Simons Investigator Award.}}

\begin{abstract}
We obtain structural results on translational tilings of periodic functions in $\Z^d$ by finite tiles. In particular, we show that any level one tiling of a periodic set in $\Z^2$ must be weakly periodic (the disjoint union of sets that are individually periodic in one direction), but present a counterexample of a higher level tiling of $\Z^2$ that fails to be weakly periodic. We also establish a quantitative version of the two-dimensional periodic tiling conjecture which asserts that any finite tile in $\Z^2$ that admits a tiling, must admit a periodic tiling, by providing a polynomial bound on the period; this also gives an exponential-type bound on the computational complexity of the problem of deciding whether a given finite subset of $\Z^2$ tiles or not.  As a byproduct of our structural theory, we also obtain an explicit formula for a universal period for all tilings of a one-dimensional tile.
\end{abstract}
\end{frontmatter}

\section{Introduction}
	\subsection{Translational tiling} Let $d \geq 1$ be an integer, and let $F \subset \Z^d$ be a finite subset of the standard lattice $\Z^d$.  A \emph{tiling}\footnote{All tilings in this paper are by translation; we do not consider tilings that involve rotation or reflection in addition to translation.} of $\Z^d$ by $F$ is a subset $A$ of $\Z^d$ with the property that every element of $\Z^d$ has precisely one representation of the form $f+a$ with $f \in F$ and $a \in A$.  We refer to $F$ as the \emph{tile} and $A$ as the \emph{tiling set}, thus $\Z^d$ is partitioned (or \emph{tiled}) by translates $F+a$ of the tile $F$ by elements $a$ of the tiling set.  In terms of the convolution operation
	$$ f*g(x) \coloneqq \sum_{y \in \Z^d} f(y) g(x-y)$$
	on $\Z^d$ (which is well-defined if at least one of $f,g$ is compactly supported, and $f,g$ are real-valued; alternatively, one of $f,g$ can take values in the unit circle $\R/\Z$ if the other is integer-valued), this property can be expressed as
	$$ \one_F * \one_A = 1,$$
	where $\one_F$ denotes the indicator function of $F$, and similarly for $A$.  More generally, for any natural number $k$ and a subset $E \subset \Z^d$, a \emph{tiling of level $k$} of $E$ by the tile $F$ is a set $A$ such that
	$$ \one_F * \one_A = k \one_E.$$
	We omit the qualifier ``of level $k$'' if $k=1$, and ``of $E$'' if $E = \Z^d$.
	
	\subsection{Periodicity} We call a function $f \colon \Z^d \to \R$ \emph{$\Lambda$-periodic} for some subgroup $\Lambda$ of $\Z^d$ if $f(x+\lambda) = f(x)$ for all $x \in \Z^d$ and $\lambda \in \Lambda$; we simply call $f$ \emph{periodic} if it is $\Lambda$-periodic for some $\Lambda$ which is a lattice (a subgroup on $\Z^d$ whose index $[\Z^d:\Lambda]$ is finite).  
	
	We call a set $E \subset\Z^d$ \emph{$\Lambda$-periodic} (resp. \emph{periodic}) if $\one_E$ is $\Lambda$-periodic (resp. periodic).  Note from Lagrange's theorem that if $\Lambda$ is a lattice of index $\ell \coloneqq [\Z^d:\Lambda]$, then any $\Lambda$-periodic set or function will also be $\ell\Z^d$-periodic.  We call a set $E\subset \Z^2$ \emph{weakly periodic} if it can be represented as the disjoint union $E = E_1 \uplus \dots \uplus E_m$ of finitely many sets $E_1,\dots,E_m$, with each $E_j$ $\langle h_j\rangle$-periodic along a one-dimensional subgroup $\langle h_j \rangle \coloneqq \{n h_j \colon n \in \Z\}$ for some non-zero $h_j \in \Z^2$.
	
	\subsection{Periodicity of tiling}\label{sub-periodicity} In one dimension it is easy to see from the pigeonhole principle that any tiling $A$ by a finite tile, of any level $k$, is periodic (see \cite{N}).  However in higher dimensions tiling sets need not be periodic.  For instance, if $d=2$ and $F_1$ is the square $F_1 = \{0,1\}^2$, then any set $A_1$ of the form
	\begin{equation}\label{a1-ex}
	A_1 \coloneqq \{ (2n, 2m+a(n)) \colon n,m \in \Z \}
	\end{equation}
	with $a \colon \Z \to \{0,1\}$ an arbitrary function, is a tiling by $F_1$, but is only periodic if $a$ is periodic.  On the other hand, we observe that this tiling is still $\langle (0,2)\rangle$-periodic.  As a slightly more complex example of this type, if $F_2 \coloneqq \{0,2\} \times \{0,1\}$, then any set $A_2$ of the form
	\begin{equation}\label{a2-ex}
	A_2 \coloneqq \{ (4n, 2m+a(n)) \colon n,m \in \Z \} \cup \{ (4n+1+2b(m), 2m) \colon n \in \Z \}
	\end{equation}
	for arbitrary functions $a,b \colon \Z \to \{0,1\}$ 
	is a tiling by $F_2$.  In general, $A_2$ will not be periodic along any non-trivial group $\Lambda$, but will always be weakly periodic, being the disjoint union of a $\langle (0,2)\rangle$-periodic set and a $\langle (4,0) \rangle$-periodic set.
	
	Our investigations were primarily motivated by the \emph{periodic tiling conjecture} of Lagarias and Wang \cite{LW} (which also implicitly appears previously in \cite[p. 23]{GS}):
	
	\begin{conj}[Periodic tiling conjecture]\label{periodic-conj}  Let $F \subset \Z^d$ be a finite tile.  If there is at least one tiling $A$ by $F$, then there exists a tiling $A'$ by $F$ which is periodic.
	\end{conj}
	
	This conjecture was established in $d=1$ in \cite{N} as a quick application of the pigeonhole principle.  For $d=2$, the conjecture was recently established by Bhattacharya \cite{BH} using ergodic theory techniques and a ``dilation lemma'' proven using elementary number theory (or elementary commutative algebra); see \cite{WV} for some earlier partial results in the $d=2$ case.  For $d>2$ the conjecture is known when the cardinality $|F|$ of $F$ is prime or equal to $4$ \cite{szegedy}, but remains open in general.  On the other hand, the tiling conjecture for multiple tiles $F_1,\dots,F_k$ in $\Z^d$ is known to be false \cite{B}, \cite{R}.  Finally, we remark that by a well known argument attributed to Wang (see \cite{B}, \cite{R}), the validity of Conjecture \ref{periodic-conj} at a given dimension $d$ implies that the problem of determining whether a given tile $F$ tiles $\Z^d$ or not is decidable. We refer the reader to \cite{R}, \cite{szegedy} for further discussion and surveys of tiling problems in lattices.
	
	\begin{remark}  There is also extensive literature on tiling problems on other groups than $\Z^d$, both by indicator functions $\one_F$ and by more general tiling functions $f$. For instance, the analogue of Conjecture \ref{periodic-conj} in $\R^d$ is known for convex polytopes \cite{V}, \cite{M, Mc81} and for topological disks \cite{gn}, \cite{kenyon}, and the one-dimensional case is established (for bounded tiles) in \cite{leptin-muller}, \cite{LW}, \cite{kl}.  See the recent papers \cite{liu}, \cite{kl-survey} for further results and open problems of tiling in $\R$ and in $\R^d$.  Tiling of more general locally compact groups by functions is studied in \cite{HN}, \cite{leptin-muller}.  There is also substantial literature on tiling finite abelian groups, which in this context is also known as \emph{factorization}; see the text \cite{ss} for a detailed presentation of this topic. Finally, we note some connections between discrete tilings and low-complexity subshifts of finite type, see for instance the recent survey \cite{k-survey}. However, the focus of this paper will be exclusively on tiling problems in lattices $\Z^d$.
	\end{remark}

	\subsection{Results} We can now state our first main theorem, which clarifies the nature of  level one tilings of periodic sets  in two dimensions, in particular revealing a fundamental difference between level one tilings and higher level tilings.
	
	\begin{thm}[Tilings in $\Z^2$]\label{main}\ 
		\begin{enumerate-math}
			\item (Level one tilings in the plane are weakly periodic) If $F \subset \Z^2$ is finite and $A$ is a level one tiling of a periodic set by $F$, then $A$ is weakly periodic. 
			\item (Higher level tilings in the plane need not be weakly periodic)  There exists an eight-element subset $F \subset \Z^2$ and a level $4$ tiling of $\Z^2$ which is not weakly periodic.
		\end{enumerate-math}
	\end{thm}
	
	Theorem \ref{main}(ii) is established by an explicit construction which we give in Section \ref{counter-sec}; the tiling set is a finite Boolean combination of ``Bohr sets''.  We now discuss Theorem \ref{main}(i).
	In fact, we have a more quantitative version of this result.  Here and in the sequel we use the asymptotic notation $X = O_{|F|}(Y)$, $X \ll_{|F|} Y$, or $Y \gg_{|F|} X$ to denote an estimate of the form $|X| \leq C_{|F|} Y$ where $C_{|F|}$ is a quantity depending only on $|F|$.  Similarly with the subscript $|F|$ replaced by other parameters.
	
	\begin{thm}[Quantitative weak periodicity of level one  tilings in $\Z^2$]\label{main-quant}  Let $F \subset \Z^2$ be a tile with $1 < |F| < \infty$ and $0 \in F$, and let $A$ be a level one tiling of an $\ell \Z^2$-periodic set $E$ by $F$ for some natural number $\ell$.  Then  there is a lattice $\Lambda \subset \ell \Z^2$ of index 
		$$ [\ell \Z^2 : \Lambda] \ll_{|F|} \mathrm{diam}(F)^{2(|F|-1)^2} $$ 
		and pairwise incommensurable vectors $h_1,\dots,h_m \in \Z^2$ for some $1 \leq m \leq |F|-1$ with magnitude bounds
		\begin{equation}\label{h-b}
		\|h_1\|, \dots, \|h_m\| \leq \mathrm{diam}(F)^{|F|-1}
		\end{equation}
		and a positive integer multiple $L$ of size
		\begin{equation}\label{l-b}
		L \ll_{|F|} \mathrm{diam}(F)^{|F|(|F|-1)}
		\end{equation}
		such that the intersection of $A$ with each coset of $\Lambda$ is $\langle \ell L h_j \rangle$-periodic for some $j=1,\dots,m$.  Furthermore each $h_j$ is an integer multiple of a vector in $F \backslash \{0\}$.  The lattice $\Lambda$, integer $L$, and the vectors $h_1,\dots,h_m$ are allowed to depend on $E,\ell,F$ but do not otherwise depend on the choice of tiling set $A$.
	\end{thm}

	Note that Theorem \ref{main-quant} implies Theorem \ref{main}(i) after translating the tile $F$ so that $0 \in F$ and dealing with the easy case $|F|=1$ separately. \thmref{main-quant} also allows us to classify two-dimensional tilings of level one in terms of one-dimensional tilings. This classification, in turn, provides a  description of the structure of any two-dimensional tiling of level one.
	Moreover, using Theorem \ref{main-quant} we prove the following quantitative generalization of the $d=2$ case  of Conjecture \ref{periodic-conj}, with polynomial type bounds (and also an extension to tilings of other periodic subsets $E$ of $\Z^2$ than the full lattice $\Z^2$):
	
	\begin{thm}[Quantitative periodic tiling conjecture in two dimensions]\label{period-quant}  Let $F \subset \Z^2$ be a finite tile, and let $E \subset \Z^2$ be an $\ell \Z^2$-periodic set for some $\ell \geq 1$.  If there is at least one tiling $A$ of $E$ by $F$, then there exists a tiling $A'$ of $E$ by $F$ by an $\ell M \Z^2$-periodic set with
		$$ M \ll_{|F|} \mathrm{diam}(F)^{O(|F|^4)}.$$
	\end{thm}
	
	This theorem has the following bound on the computational complexity of deciding whether a given finite set $F$ is a tile, which is of exponential type in the diameter of $F$ if $|F|$ is bounded:
	
	\begin{corollary}[Computational complexity bound for planar tiling]\label{ccb}  There is an algorithm which, when given a finite subset $F$ of $\Z^2$ as input, decides whether $F$ can tile $\Z^2$ in time $O_{|F|}( \exp( O_{|F|}( \mathrm{diam}(F)^{O(|F|^4)} ) ) )$ (counting each arithmetic operation as costing time $O(1)$).
	\end{corollary}
	
	\begin{proof}  By Theorem \ref{period-quant}, it suffices to check tilings $A$ that are $M\Z^2$-periodic for some $M \ll_{|F|} \mathrm{diam}(F)^{O(|F|^4)}$.  Each such tiling can be checked in time $O_{|F|}(M^2)$, and the number of $M\Z^2$-periodic tilings is at most $2^{M^2}$ for each $M$, giving the claim.
	\end{proof}
	
	Note that the results in \cite{BH} established that the tiling problem in $\Z^2$ was decidable, but gave no bound on the computational complexity. The proof of Corollary \ref{ccb} also shows that for fixed $|F|$, the decision problem is in the complexity class $NP$ with respect to the diameter $\mathrm{diam}(F)$; for instance, in the unlikely event that $P=NP$, the decision problem could now be performed in time polynomial in the diameter.  It seems of interest to see if the exponential bound can be improved without the $P=NP$ hypothesis in the regime when $|F|$ is bounded.  In one dimension the best complexity bound currently known is $\exp(O_\eps(\mathrm{diam}(F)^{1/3+\eps}))$ for any $\eps>0$, due to B\'ir\'o \cite{biro}.
	
	We now discuss further the proof of Theorem \ref{main-quant}.  Our starting point is the following quite explicit structural theorem valid in all dimensions and levels, which we believe to be of independent interest:
	
	\begin{thm}[Structure of tilings]\label{structure}  Let $d \geq 1$, let $F$ be a finite subset of $\Z^d$, and let $A \subset \Z^d$ be a set such that $\one_F * \one_A$ is $\ell \Z^d$-periodic for some $\ell \geq 1$.  We normalize $0 \in F$.  Then there exists a decomposition
		\begin{equation}\label{afd}
		\one_A = \one_F * \one_A - \sum_{f \in F \backslash \{0\}} \varphi_f
		\end{equation}
		where for each $f \in F \backslash \{0\}$, $\varphi_f \colon \Z^d \to [0,1]$ is a function which is $\langle qf\rangle$-periodic, where $q$ is the least common multiple of $\ell$ and all the primes less than or equal to $2|F|$.
	\end{thm}
	
	We establish this result in Section \ref{dilation-sec}.  It is a consequence of (a generalization of) the dilation lemma from \cite{BH}, which roughly speaking asserts that if a set (or multi-set) $A$ is a tiling for a tile $F$, then it is also a tiling for all dilations $rF$ of that tile, as long as $r$ is congruent to $1$ with respect to a suitable modulus.  This fact is number-theoretic in nature, ultimately boiling down to the Frobenius identity $(x+y)^p = x^p + y^p$ that is valid in any commutative ring of characteristic $p$.  Theorem \ref{structure} is then established by averaging over all such dilations $r$. 
	
	\begin{remark} Theorem \ref{structure} has the Fourier-analytic consequence that the distributional Fourier transform $\ft\one_A$ of $A$, which is a distribution on the torus $(\R/\Z)^d$, is supported on the union of the finite subgroup $\left(\frac{1}{\ell} \Z/\Z \right)^d$ and the codimension one subgroups $(qf)^\perp \coloneqq \{ \xi \in (\R/\Z)^d \colon qf \cdot \xi =0\}$ for $f \in F \backslash \{0\}$.  A qualitatively similar conclusion\footnote{We also note a possibly related result of Granville and Rudnick \cite[Corollary 3.1]{GR}, which in our notation states that if $f \colon \Z^d \to \Z$ is compactly supported and not identically zero (e.g., if $f = \one_F$ for some finite non-empty $F \subset \Z^d$), then the set $\{ \xi \in (\Q/\Z)^d: \hat f(\xi)=0\}$ of roots of unity in the zero set of $\hat f$ can be placed inside the union of a finite number of explicitly computable codimension one subgroups; this implies a special case of a conjecture of Lang \cite{Lang}.  We thank the anonymous referee for this reference.} regarding the spectral measure of a measure-preserving system associated to a tiling was obtained in \cite[Lemma 3.2]{BH}.  Our initial arguments relied heavily on this Fourier analytic structure, but we found eventually that physical-space arguments were simpler and gave superior bounds to those relying on the Fourier transform. Furthermore, the physical-space approach we developed provided us with more structural data, which in particular allowed us to gain better understanding of the rigidity of  tiling structures in $\Z^2$.
	\end{remark}
	
	Theorem \ref{structure} already resembles an assertion of weak periodicity of $\one_A$, except that the terms on the right-hand side of \eqref{afd}
	are not indicator functions.  Nevertheless, the structural theorem turns out to be particularly powerful in the case of level one tilings, when it imposes a powerful pointwise constraint 
	\begin{equation}\label{const}
	\sum_{f \in F \backslash \{0\}} \varphi_f \leq 1
	\end{equation}
	on the functions $\varphi_f$.  Furthermore, by working modulo $1$ to eliminate the $\one_A$ and $\one_F * \one_A$ terms from \eqref{afd}, we also have the important identity
	$$ \sum_{f \in F \backslash \{0\}} \varphi_f = 0 \hbox{ mod } 1.$$
	In dimension two, one can apply discrete differentiation operators, exploiting the partial periodicity of the $\varphi_f$ to conclude that the functions $\varphi_f \hbox{ mod } 1$ are polynomials (after collecting commensurable terms).  The classical Weyl equidistribution theory of these polynomials then asserts that these functions are either periodic or equidistributed in the unit circle.  The powerful constraint \eqref{const} lets us eliminate the equidistributed case, and some further elementary arguments (involving linear algebra facts such as Cramer's rule) then allow us to conclude Theorem \ref{main-quant}.
	
	The proof of Theorem \ref{period-quant} also proceeds by exploiting the dilation lemma, though in a simpler fashion (there is no need to work modulo $1$ in this case).  One could also establish results similar to Theorem \ref{period-quant} but with weaker bounds (of exponential type in the diameter rather than polynomial) by using the pigeonhole principle instead of the dilation lemma, but we do not present these arguments here.
	
	As a further application of our structural results, we establish in Corollary \ref{structure-1d} an explicit formula for a universal period for all tilings $A$ of a given one-dimensional tile $F \subset \Z$, which (remarkably) is only of polynomial size in the diameter, as opposed to exponential, in the regime where the cardinality $|F|$ of the tile is bounded.
	
	Our results leave open the question of whether the analogue of Conjecture \ref{periodic-conj} for higher level (i.e., whether any tile $F$ that admits a level $k$ tiling, also admits a level $k$ tiling by a periodic set) is true in two dimensions; neither our positive or negative results seem strong enough to resolve this question.  In the one dimensional lattice the claim easily follows from the pigeonhole principle (or from Corollary \ref{structure-1d} below), which forces all tiling sets at any level to be periodic.  On the other hand, on the continuous line $\R$ an example was given in \cite{k-l} of an $L^1(\R)$ function of unbounded support that tiled $\R$ by a set which was not the finite union of periodic sets.  We also mention the recent result of Liu \cite{liu} that if a function $f \in L^1(\R^d)$ tiles by a finite union of lattices at some level, then it also tiles by a single lattice at a possibly different level. We refer the reader to the recent survey \cite{kl-survey} for further discussion of tiling results in $\R$ and $\R^d$.
	
	We plan to investigate the applications of this theory to higher-dimensional lattice tilings in subsequent work.
	
	\subsection{Organization of the paper}

	In  \secref{counter-sec}, we present a constructive proof of \thmref{main}(ii). This section is self-contained and will not be used in what follows.  
	In \secref{dilation-sec}, we develop an approach to study tiling structures in $\Z^d$ and prove our structure theorem, \thmref{structure}. The proof relies on  \lemref{dilation}, which we refer to as the ``dilation lemma''. As a direct corollary of our structure theorem, we obtain  an explicit universal period of one-dimensional tilings, of polynomial  size in the diameter of the tile.
	In Sections \ref{weak-sec}, \ref{periodic-sec} and \ref{oneper-sec}, we apply the results of \secref{dilation-sec} to level one tilings of periodic sets  in $\Z^2$.  In more detail:
	
	 In \secref{weak-sec}, using  polynomial sequences (based on \cite{BH}), we prove  \thmref{main-quant}, which is a quantitative version of \thmref{main}(i). 
	In \secref{periodic-sec}, we prove \thmref{period-quant} by combining the former presented results with a ``slicing lemma''. 
	In \secref{oneper-sec}, by combining \thmref{main-quant} and the results in \secref{periodic-sec}, we establish a satisfactory description of the structure of any tiling of level one of a periodic set in $\Z^2$. In particular, we show that the question whether a tile $F\subset\Z^2$ admits non one-periodic tiling, is decidable.

	\subsection{Notation}
	
	If $F \subset \Z^d$ is a set and $r$ is a natural number, we use $rF \coloneqq \{rf: f \in F \}$ to denote the dilation of $F$ by $r$.  In particular we have the lattices $r \Z^d = \{ r n: n \in \Z^d \}$.  
	We use $|F|$ to denote the cardinality of $F$; by abuse of notation we also use $|z|$ to denote the magnitude of a real or complex number $z$. 
	
	We use $x \mapsto x \hbox{ mod } 1$ to denote the projection homomorphism from $\R$ to the (additive) unit circle $\R/\Z$.  A \emph{polynomial of degree at most $d$} from $\Z$ to $\R/\Z$ is a map $P \colon \Z \to \R/\Z$ of the form
	$$ P(n) = \alpha_0 + \alpha_1 n + \dots + \alpha_d n^d$$
	for some $\alpha_0,\dots,\alpha_d \in \R/\Z$.
	
	Two vectors $h_1,h_2 \in \Z^d$ are said to be \emph{commensurable} if one is a scalar multiple of the other, and \emph{incommensurable} otherwise.  In the two-dimensional case $d=2$, we define the wedge product $h_1 \wedge h_2$ to be the determinant of the $2 \times 2$ matrix with rows $h_1,h_2$; thus $h_1,h_2$ are commensurable if and only if $h_1 \wedge h_2 = 0$.  If $h_1,h_2$ are incommensurable, we observe the \emph{Cramer rule}
	\begin{equation}\label{cramer}
	v = \frac{v \wedge h_2}{h_1 \wedge h_2} h_1 + \frac{v \wedge h_1}{h_2 \wedge h_1} h_2
	\end{equation}
	for any $v \in \Z^2$ (this is easily verified by first testing the cases $v=h_1,h_2$, then extending by linearity).  In particular, if we let $\langle h_1,h_2\rangle$ denote the lattice generated by $h_1,h_2$, we have the inclusion
	\begin{equation}\label{cramer-include}
	|h_1 \wedge h_2| \Z^2 \subset \langle h_1, h_2 \rangle.
	\end{equation}
	This inclusion can also be established from Lagrange's theorem, since $|h_1 \wedge h_2|$ is the index of $\langle h_1,h_2 \rangle$.  From Hadamard's inequality one has
	$$ |h_1 \wedge h_2| \leq \|h_1\| \|h_2\|$$
	where $\|h\|$ denotes the Euclidean length of an element $h$ of $\Z^d$.
	
	An element $h \in \Z^d$ is said to be \emph{primitive} if it cannot be written as $h = mh'$ for some $h' \in \Z^d$ and some integer $m>1$.  Note that every non-zero element $h$ of $\Z^d$ can be uniquely expressed as $h = mh'$ where $m$ is a natural number and $h' \in \Z^d$ is primitive.
	
	If $h \in \Z^d$, we use $\delta_h = \one_{\{h\}}$ to denote the Kronecker delta function at $h$, and let $\Delta_h$ denote the discrete differentiation operator
	$$ \Delta_h f(x) \coloneqq (\delta_0 - \delta_h) * f(x) = f(x) - f(x-h)$$
	in the direction $h$ applied to a function  $f \colon \Z^d \to \R$ (or $f \colon \Z^d \to \R/\Z$) at a location $x \in \Z^d$. Note that these operators $\Delta_h$ commute with each other and with convolution by any additional function $g$: 
	$$\Delta_h (f*g) = (\Delta_h f) * g = f * \Delta_h g.$$

	\section{The counterexample}\label{counter-sec}
	
	In this section we prove Theorem \ref{main}(ii).  The constructions here are not used elsewhere in the paper; however, the analysis presented in the rest of the paper was what led us to the counterexample presented here.
	
	We first need to locate a sign pattern on $\Z^2$ that obeys certain cancellation properties.
	
	\begin{lem}\label{chi-lem}
		There exists a $4\Z^2$-periodic function $\chi:\Z^2\to\{-1,1\}$ such that one has the cancellations
		\begin{align}
		\chi * \one_{\{(0,0), (0,2)\}} &= 0 \label{chi-1} \\
		\chi * \one_{\{(0,0), (1,0)\}} &= 0. 
		\label{chi-2} 
		\end{align}
	\end{lem}
	
	\begin{proof}
		It is a routine matter to check that the function
		\begin{equation}
		\chi(m_1,m_2) \coloneqq (-1)^{\lfloor m_2/2 \rfloor + m_1}
		\end{equation} 
		satisfies the claimed properties, where $\lfloor\cdot \rfloor$ is the integer part function.
	\end{proof}
	
	Note that \eqref{chi-1} and \eqref{chi-2} imply 
	\begin{equation}\label{chi-3}
	    \chi * \one_{\{(0,0), (2,-2)\}} = 0
	\end{equation}
	since
	$$\one_{\{(0,0), (2,-2)\}} = \one_{\{(0,0), (0,2)\}} * \one_{\{(0,-2)\}} + \one_{\{(0,0), (1,0)\}} * (\one_{\{(1,-2)\}} - \one_{\{(0,-2)\}}).$$
	Let $\alpha$ be an arbitrary irrational number (e.g., $\alpha = \sqrt{2}$). Consider the function\footnote{The function $(m_1,m_2) \mapsto \{\alpha m_1\}+\{\alpha m_2\}-\{\alpha(m_1+m_2)\}$ was also mentioned in a closely related context in \cite[Example 4]{ks}, \cite[Example 3.4.6]{szabados}.}
	$$a(m_1,m_2) \coloneqq \chi(m_1,m_2)\left(\{\alpha m_1\}+\{\alpha m_2\}-\{\alpha(m_1+m_2)\}-1/2 \right) + 1/2$$
	on $\Z^2$, where $\chi$ is as in Lemma \ref{chi-lem} and $\{x\} = x - \lfloor x \rfloor \in [0,1)$ denotes the fractional part of $x\in\R$.  Observe that $\{ \alpha(m_1+m_2)\}$ is either equal to $\{\alpha m_1\} + \{\alpha m_2\}$ or $\{\alpha m_1\} + \{\alpha m_2\}- 1$, hence $a(m_1,m_2)$ takes values in $\{0,1\}$.  Hence this is the indicator function of some set $A \subset \Z^2$:
	$$\one_A(m_1,m_2) \coloneqq \chi(m_1,m_2)(\{\alpha m_1\}+\{\alpha m_2\}-\{\alpha(m_1+m_2)\}-1/2) + 1/2.$$
	Now introduce the eight-element tile
	$$F \coloneqq \{t_1 (0,2) + t_2 (1,0) + t_3 (2,-2): t_1,t_2,t_3 \in \{0,1\}\}.$$
	Note that we have the factorization
	$$ \one_F = \one_{\{(0,0), (0,2)\}} * \one_{\{(0,0), (1,0)\}} * \one_{\{(0,0), (2,-2)\}}.$$
	From \eqref{chi-1} we see that the functions $$(m_1,m_2) \mapsto \chi(m_1,m_2), \quad  (m_1,m_2) \mapsto \chi(m_1,m_2) \{ \alpha m_1\}$$ are annihilated by convolution with $\one_{\{(0,0), (0,2)\}}$.  Similarly, from \eqref{chi-2} we see that $$(m_1,m_2) \mapsto \chi(m_1,m_2) \{ \alpha m_2\}$$ is annihilated by convolution with $\one_{\{(0,0), (1,0)\}}$, and from \eqref{chi-3} we see that $$(m_1,m_2) \mapsto \chi(m_1,m_2) \{ \alpha(m_1+m_2)\}$$ is annihilated by convolution with $\one_{\{(0,0), (2,-2)\}}$.  Finally, we have $\one_F * 1 =\one_F*\one_{\Z^2}= |F| = 8$.  Using the bilinear, commutative, and associative properties of convolution, we conclude that
	$$ \one_F * \one_A = 4.$$
	
	It remains to show that $A$ is not weakly periodic.  Suppose for contradiction that we had a decomposition
	\begin{equation}\label{onea}
	\one_A = \one_{A_1} + \dots + \one_{A_m}
	\end{equation}
	where each $A_i$ is $\langle h_i \rangle$-periodic for some non-zero $h_i \in \Z^2$.  By repeatedly grouping together sets $A_i$ corresponding to commensurate $h_i$ and passing to least common multiples, we may assume that the $h_i$ are pairwise incommensurable.  In particular, at most one of the $h_i$ is commensurate with $(1,0)$; by relabeling we may assume that $h_i$ is incommensurate with $(1,0)$ for all $2 \leq i \leq m$; by adding a dummy index if necessary (and an empty set $A_1$) we may assume that $h_1$ is commensurate with $(1,0)$.  We now rearrange \eqref{onea} as
	\begin{equation}\label{chi-mm}
	\begin{split}
	\chi(m_1,m_2) \{ \alpha m_2 \} - \one_{A_1}(m_1,m_2)& = \sum_{i=2}^m \one_{A_i}(m_1,m_2) - \chi(m_1,m_2) \{ \alpha m_1\} \\
	&\quad + \chi(m_1,m_2) \{ \alpha (m_1+m_2) \} + \tfrac{1}{2} \chi(m_1,m_2) - \tfrac{1}{2}.
	\end{split}
	\end{equation}
	The two terms on the left-hand side of \eqref{chi-mm} are $\langle (4,0) \rangle$-periodic and $\langle h_1 \rangle$-periodic respectively, thus the left-hand side of \eqref{chi-mm} is $\langle 4h_1 \rangle$-periodic. Meanwhile, each of the terms on the right-hand side of \eqref{chi-mm} is $\langle h \rangle$-periodic with respect to some $h$ incommensurate with $(1,0)$ and hence with  $4h_1$.   Let $\tilde e \in \Z^2$ be a vector incommensurate with all of these periods $h$.  Then we may find an integer multiple $e = N \tilde e$ of $\tilde e$ which lies in $\langle 4h_1, h \rangle$ for all the periods $h$ on the right-hand side of \eqref{chi-mm}, thus one has a decomposition
	$$ e = a_h (4h_1) + b_h h$$
	for each such $h$.  Applying the discrete differentiation operator $\Delta_{e - a_h(4h_1)}$ then annihilates any term on the right-hand side of \eqref{chi-mm} that is $\langle h \rangle$-periodic, and the operator is equivalent to $\Delta_e$ when applied to the left-hand side of \eqref{chi-mm}.  Applying enough of these discrete differentiation operators to annihilate the entire right-hand side of \eqref{chi-mm}, we conclude that
	$$ \Delta_e^k ( \chi(m_1,m_2) \{ \alpha m_2 \} - \one_{A_1}(m_1,m_2) ) = 0$$
	for some integer $k$.  Thus, when evaluated on any coset of $\langle e \rangle=\{ne:n\in\Z\}$, the $k^{\mathrm{th}}$ discrete derivative of function 
	\begin{equation}\label{fun}
	\chi(m_1,m_2) \{ \alpha m_2 \} - \one_{A_1}(m_1,m_2)
	\end{equation} 
	vanishes. A simple induction on $k$ then shows that \eqref{fun} is a polynomial (of degree at most $k-1$); it is also bounded, hence it is constant. In other words, \eqref{fun} is $\langle e \rangle$-periodic.  As it is also $\langle 4h_1 \rangle$-periodic, it is in fact periodic, and thus attains at most finitely many values.  But this implies that $\{ \alpha m_2 \}$ attains at most finitely many values, contradicting the irrationality of $\alpha$.  This concludes the proof of Theorem \ref{main}(ii).
	
	\begin{remark}  Since $\{(0,0), (0,2)\}$ admits a periodic tiling of level $1$, $F$ admits a periodic tiling of level $4$.  Hence this example does not provide a counterexample to the higher level version of Conjecture \ref{periodic-conj}, which remains open even in two dimensions.
	\end{remark}
	
	\section{A dilation lemma and the structure theorem}\label{dilation-sec}
	
	In \cite[Proposition 3.1]{BH}, some elementary commutative algebra was used to establish a dilation lemma that  asserted, roughly speaking, that if $A$ was a (multi-set) tiling of $\Z^d$ of a tile $F$ at some level $k$, then $A$ was also a tiling of the dilate $rF$ for an arithmetic progression of $r$'s.  A one-dimensional version of this lemma previously appeared in \cite{tijdeman}; see also \cite[Proposition 3.2]{hiprv} (or \cite[Theorem 3.3]{imp}) for a related Fourier-analytic dilation lemma for tilings of $\mathbb{F}_p^d$.  Variants of this lemma were also established in \cite[Lemma 10]{szegedy} and \cite[Lemma 2]{ks}.  We re-prove this lemma using elementary number theory, and generalize it from tilings of $\Z^d$ to tilings of periodic level functions.
	
	\begin{lem}[Dilation lemma]\label{dilation}  Let $F$ be a finite subset of $\Z^d$ for some $d \geq 1$, and let $g \colon \Z^d \to \Z$ be a bounded function.
		\begin{enumerate-math}
			\item  If $\one_F * g = k$ for some integer $k$, then for any prime $p$ with $p > (\sup g - \inf g) |F|$, one has $\one_{pF} * g = k$.
			\item If $\one_F * g = k$ for some integer $k$, and $q$ is the product of all primes less than or equal to $(\sup g - \inf g) |F|$, then one has $\one_{rF} * g = k$ whenever $r$ is a natural number coprime to $q$.
			\item  If $\one_F * g$ is $\ell \Z^d$-periodic for some $\ell \geq 1$, and $q$ is the least common multiple of $\ell$ and all primes less than or equal to $2(\sup g - \inf g) |F|$, then $\one_{rF} * g = \one_F * g$
			whenever $r$ is a natural number with $r = 1 \hbox{ mod } q$.
		\end{enumerate-math}
	\end{lem}
		
	\begin{proof}  We begin with (i).  The claim is easily verified when $g$ is constant, so we may assume that $g$ is non-constant, in particular $p > |F|$.  We convolve the equation $\one_F * g = k$ by $p-1$ further copies of $\one_F$ using the identity $ \one_F*1 = |F|$ to conclude that
		$$ (\one_F)^{*p} * g = |F|^{p-1} k$$
		where $(\one_F)^{*p}$ is the convolution of $p$ copies of $\one_F$.  As all functions here are integer-valued, this identity also holds modulo $p$:
		$$(\one_F)^{*p} * g = |F|^{p-1} k \hbox{ mod } p.$$
		By Fermat's little theorem we have $|F|^{p-1} = 1 \hbox{ mod } p$.  Also, from the binomial theorem\footnote{Alternatively, one can apply the Frobenius endomorphism $f \mapsto f^{*p}$ to the group algebra $\F_p[\Z^d]$, where $\F_p$ denotes the finite field of order $p$.} we have $(f+g)^{*p} = f^{*p} + g^{*p} \hbox{ mod } p$ for all finitely supported functions $f,g \colon \Z^d \to \Z$.  Iterating this observation and writing $\one_F = \sum_{f \in F} \delta_f$ as the sum of Kronecker delta functions $\delta_f$, we see that
		$$ (\one_F)^{*p} = \one_{pF} \hbox{ mod } p.$$
		We conclude that
		$$ \one_{pF} * g = \one_F * g \hbox{ mod } p$$
		or equivalently
		$$ (\one_{pF}-\one_F) * g = 0 \hbox{ mod } p.$$
		Observe that the left-hand side takes values in the integers of magnitude at most $(\sup g - \inf g) |F|$ (to show this, one can first shift $g$ by a constant to normalize $\inf g = 0$ if desired). By the assumption on $p$, we thus see that $(\one_{pF}-\one_F) * g = 0$, and (i) follows.
		
		To prove (ii), observe from the fundamental theorem of arithmetic that any $r = 1 \hbox{ mod } q$ is the product of a finite number of primes $p > (\sup g - \inf g) |F|$ (possibly with repetition).  The claim then follows by iterating (i).
		
		Finally, we prove (iii).  If $\one_F * g$ is $\ell \Z^d$-periodic, then for any $\lambda \in \ell \Z^d$ we have
		$$ \one_F * \Delta_\lambda g = \Delta_\lambda( \one_F * g ) = 0.$$
		The discrete derivative $\Delta_\lambda g$ takes values in the integers of magnitude at most $\sup g - \inf g$.  Applying (ii), we conclude that
		$$ \one_{rF} * \Delta_\lambda g = 0$$
		whenever $\lambda \in \ell \Z^d$ and $r = 1 \hbox{ mod } q$.  Equivalently, $\one_{rF} * g$ is $\ell \Z^d$-periodic for all $r = 1 \hbox{ mod } q$, which implies that
		$$ (\one_{rF} - \one_F) * g$$
		is $\ell \Z^d$-periodic.  In particular, if $B_R$ denotes a F{\o}lner sequence on $\ell \Z^d$ (for instance one can take $B_R \coloneqq \ell \{-R,\dots,R\}^d$), then we have
		\begin{equation}\label{br}
		(\one_{rF} - \one_F) * g = (\one_{rF} - \one_F) * \frac{1}{|B_R|} \one_{B_R} * g
		\end{equation}
		for any $R>0$.  But as $r-1$ is a multiple of $\ell$, we have $rf - f \in \ell \Z^d$ for all $f \in F$.  From the F{\o}lner property we then have
		$$ \lim_{R \to \infty} \left\| (\delta_{rf} - \delta_f) * \frac{1}{|B_R|} \one_{B_R} \right\|_{\ell^1(\Z^d)} = 0$$
		for all $f$; from Young's inequality and the boundedness of $g$ we hence have
		$$ \lim_{R \to \infty} \left\| (\delta_{rf} - \delta_f) * \frac{1}{|B_R|} \one_{B_R} * g \right\|_{\ell^\infty(\Z^d)} = 0$$
		and hence by the triangle inequality
		$$ \lim_{R \to \infty} \left\| (1_{rF} - 1_F) * \frac{1}{|B_R|}  \one_{B_R} * g \right\|_{\ell^\infty(\Z^d)} = 0.$$
		Combining this with \eqref{br} one has $ (\one_{rF} - \one_F) * g=0$, which gives (iii).
	\end{proof}
	
	\begin{remark} The above proof shows that if the requirement $r = 1 \hbox{ mod } q$ in Lemma \ref{dilation}(iii) is relaxed to $r$ merely being coprime to $q$, then $\one_{rF} * g$ is no longer necessarily equal to $\one_F * g$, but will still be $\ell \Z^2$-periodic.  In \cite[Theorems 10, 13]{szegedy} it is also shown that if $\one_F * \one_A = 1$ then $\one_{-F} * \one_A = 1$, and also $\one_{rF} * \one_A \leq 1$ whenever $r$ is an integer coprime to $|F|$. (In fact one can improve the inequality $\one_{rF} * \one_A \leq 1$ to $\one_F * \one_A = 1$ by volume packing arguments; see, e.g., \cite[Lemma 3.2]{GL2}.)
	\end{remark}
	
	Now we can prove Theorem \ref{structure}.  
	
	\begin{proof}[Proof of \thmref{structure}] Let the notation and hypotheses be as in that theorem.  From Lemma \ref{dilation}(iii) we have
	$$ \one_{rF} * \one_A = \one_F * \one_A$$
	for all natural numbers $r$ with $r = 1 \hbox{ mod } q$, where $q$ being as in \lemref{dilation}(iii).  As $0 \in F$, we can rewrite this identity as
	$$ \one_A = \one_F * \one_A - \sum_{f \in F \backslash \{0\}} \delta_{rf} * \one_A.$$
	We can average this to obtain
	$$ \one_A = \one_F * \one_A - \sum_{f \in F \backslash \{0\}} \varphi_{f,N}$$
	for any natural number $N \geq 1$, where $\varphi_{f,N} \colon \Z^d \to [0,1]$ is the function
	$$ \varphi_{f,N} \coloneqq \frac{1}{N} \sum_{n=1}^N \delta_{(1+nq)f} * \one_A.$$
	It is clear that $\varphi_{f,N}$ takes values in $[0,1]$. Also from telescoping series we have
	\begin{equation}\label{vfn}
	|\varphi_{f,N}(x+qf) - \varphi_{f,N}(x)| \leq \frac{2}{N}
	\end{equation}
	for any $x \in \Z^d$.  By the Arzel\'a-Ascoli theorem, we can find a sequence $N_i \to \infty$ such that for every $f \in F \backslash \{0\}$, $\varphi_{f,N_i}$ converges locally uniformly to a limit $\varphi_f$, which then also takes values in $[0,1]$, and we now have
	$$ \one_A = \one_F * \one_A - \sum_{f \in F \backslash \{0\}} \varphi_{f}.$$
	Setting $N=N_i$ in \eqref{vfn} and taking limits, we conclude that $\varphi_f(x+qf)-\varphi_f(x)=0$ for all $x \in \Z^d$, thus $\varphi_f$ is $\langle qf \rangle$-periodic, and Theorem \ref{structure} follows.
	\end{proof}
	
	\begin{remark}  The above argument shows that one can interpret $\varphi_f(x)$ as a limiting density of $A$ along the ray $\{ x + (1+nq)f \colon n \in \N \}$.  A similar averaging argument (using ultrafilter limits instead of generalized limits) appears in \cite[Section 7.3]{szabados} in a related context,  under different assumptions.  
	\end{remark}
	
	For our application it is convenient to group together ``commensurable'' terms in Theorem \ref{structure}.
	
	\begin{thm}[Structure of tilings, II]\label{structure-2}  Let $d, \ell, k \geq 1$, let $F$ be a finite subset of $\Z^d$, let $E$ be an $\ell\Z^d$-periodic subset of $\Z^d$, and let $A$ be a tiling of $E$ by $F$ of level $k$.  We normalize $0 \in F$, and assume $|F|>1$.  Then there exists a decomposition
		\begin{equation}\label{vph}
		\one_A = k\one_E - \sum_{j=1}^m \varphi_j
		\end{equation}
		where $1 \leq m \leq |F|-1$, and each $\varphi_j \colon \Z^d \to [0,k]$ is $\langle qh_j \rangle$-periodic, where $q$ is the least common multiple of $\ell$ and all the primes less than or equal to $2|F|$, and $h_1,\dots,h_m$ are pairwise incommensurable elements of $\Z^d$ such that 
		\begin{equation}\label{h-prod}
		\prod_{j=1}^m \|h_j\| \leq \mathrm{diam}(F)^{|F|-1}.
		\end{equation}
		In particular, we have the upper bounds
		\begin{equation}\label{q-bound}
		q \ll_{|F|} \ell
		\end{equation}
		and
		\begin{equation}\label{level-1}
		\sum_{j=1}^m \varphi_j \leq k.
		\end{equation}
		Furthermore, each $h_j$ is an integer multiple of an element of $F \backslash \{0\}$.
	\end{thm}
	
	\begin{proof}  From Theorem \ref{structure} one has
		$$ \one_A = k\one_E - \sum_{f \in F \backslash \{0\}} \tilde \varphi_f$$
		where each $\tilde\varphi_f \colon \Z^d \to [0,1]$ is $\langle qf\rangle$-periodic.  We define an equivalence relation on $F \backslash \{0\}$ by declaring $f \sim f'$ if $f,f'$ are commensurable (namely,  $f = p f'$ for some rational $p$).  If we let $C_1,\dots,C_m$ be the equivalence classes of this relation, then $1 \leq m \leq |F|-1$ and we have a decomposition \eqref{vph} with
		$$ \varphi_j \coloneqq \sum_{f \in C_j} \tilde \varphi_f.$$
		In particular the $\varphi_j$ are non-negative, and then from \eqref{vph} and the non-negativity of $\one_A$ we conclude that all the $\varphi_j$ are also bounded by $k$, as well as the bound \eqref{level-1}.  Since each $\tilde \varphi_f$ is $\langle qf \rangle$-periodic, we see on taking least common multiples that $\varphi_j$ is $\langle qh_j \rangle$-periodic for some non-zero $h_j \in \Z^2$ commensurable to the elements of $C_j$ and of magnitude at most
		$$ \|h_j\| \leq \prod_{f \in C_j} \|f\| \leq \mathrm{diam}(F)^{|C_j|}$$
		(note that $\|f\| \leq \mathrm{diam}(F)$ for all $f \in F$ since $0 \in F$).  In particular we have \eqref{h-prod}.  Since the $C_j$ are pairwise incommensurable, the $h_j$ are also pairwise incommensurable.  Finally, the bound \eqref{q-bound} is clear from definition of $q$.
	\end{proof}
	
	We have found that the bound \eqref{level-1} is particularly powerful in the level one case $k=1$, as it can be used in that case to completely rule out ``equidistributed'' scenarios in which at least one of the $\varphi_j$ has values that equidistribute in the unit interval $[0,1]$.  However, in higher level settings it is possible for multiple equidistributed $\varphi_j$ to coexist, which is what led us to the counterexample constructed in Section \ref{counter-sec}.
	
	As another quick application of Theorem \ref{structure}, we obtain an explicit formula for a universal period of one-dimensional tiles:
	
	\begin{corollary}[Universal period of one-dimensional tiles]\label{structure-1d}  Let $F$ be a finite subset of $\Z$, and let $A \subset \Z$ be a set such that $\one_F * \one_A$ is $\langle \ell \rangle$-periodic for some $\ell \geq 1$.  We normalize $0 \in F$.  Then $A$ is $\langle qn \rangle$-periodic, where $q$ is the least common multiple of $\ell$ and all the primes less than or equal to $2|F|$, and $n$ is the least common multiple of $\|f\|$ for all $f \in F \backslash \{0\}$.
	\end{corollary}
	
	\begin{proof}  From Theorem \ref{structure} we see that $\one_A$ is the linear combination of a finite number of terms, each of which is $qn$-periodic.  The claim follows.
	\end{proof}

	As first observed in \cite{N}, an easy application of the pigeonhole principle already gives that all one-dimensional tilings are periodic; however, the bound produced is exponential in the diameter $\mathrm{diam}(F)$ if done naively.  In contrast, the bound here is polynomial in the diameter (for fixed $|F|$), and further is uniform over \emph{all} tilings $A$ of an $\langle \ell \rangle$-periodic set by a fixed tile $F$, whereas the period produced by pigeonhole principle arguments will depend on the choice of tiling.  In \cite{steinberger} it was shown that if the cardinality $|F|$ of the tile is not held fixed, the period of a one-dimensional tiling can grow superpolynomially in the diameter $n \coloneqq \mathrm{diam}(F)$ (in fact a lower bound $\exp( \log^2 n / 4 \log\log n )$ is demonstrated for infinitely many $n$); there is also an exponential lower bound for indecomposable tilings of higher level \cite{stein2}.   Conversely, the best known upper bound for the period for a tile of diameter $n$ (with no restriction on $|F|$) is $\exp(O_\eps(n^{1/3+\eps}))$ \cite{biro}.
	
	\begin{remark} In the special case of Corollary \ref{structure-1d} when $\one_F * \one_A = 1$, the dilation lemma of Tijdeman \cite[Theorem 1]{tijdeman} (see also \cite{cm} for an alternate proof) allows one to replace ``all the primes less than or equal to $2|F|$'' with ``all the primes dividing $|F|$''.
	\end{remark}

	\section{Weak periodicity of two-dimensional tilings of level one}\label{weak-sec}
	
	 We now prove Theorem \ref{main-quant}. Our starting point is the decomposition in Theorem \ref{structure-2}.  Accordingly, let $m, \varphi_1,\dots,\varphi_m, h_1,\dots,h_m, q$ be as in that theorem.  If $m=1$ then \eqref{vph} ensures that $A$ is $\langle qh_1 \rangle$-periodic, and we are already done.  Henceforth we assume $m \geq 2$, hence $|F| \geq 3$.
	
	Since there are at most $|F|-1$ vectors $h_1,\dots,h_m$, and they are all non-zero, one can find a vector $\tilde e \in \Z^2$ of size $O_{|F|}(1)$ which is incommensurable to all of the $h_1,\dots,h_m$.  Next, let $N$ be the least common multiple of all the determinant magnitudes $|h_i \wedge h_j|$ for $1 \leq i < j \leq m$, thus by \eqref{cramer-include} one has
	\begin{equation}\label{nlat}
	N \Z^2 \subset \langle h_i, h_j \rangle
	\end{equation}
	for all $1 \leq i < j \leq m$.  We also have the bounds
	\begin{equation}\label{N-bound}
	\begin{split}
	N &\leq \prod_{1 \leq i < j \leq m} \|h_i\| \|h_j\| \\
	&\leq (\prod_{j=1}^m \|h_j\|)^{m-1} \\
	&\leq  \mathrm{diam}(F)^{(|F|-1)(|F|-2)}.
	\end{split}
	\end{equation}
	For any $x \in \Z^2$ and $j=1,\dots,m$, we introduce the one-dimensional functions $P_{x,j} \colon \Z \to [0,1]$ by the formula
	\begin{equation}\label{Pdef}
	P_{x,j}( n ) \coloneqq \varphi_j( x + ne),
	\end{equation}
	where $e=qN\tilde e$. These functions enjoy several useful properties:
	
	\begin{prop}[Properties of $P_{x,j}$]\label{pxj-prop} Let $x \in \Z^2$. 
		\begin{enumerate-math}
			\item  One has $\sum_{j=1}^m P_{x,j}(n) = 0 \hbox{ mod } 1$ for all $n \in \Z$.
			\item  For each $1 \leq j \leq m$, the map $n \mapsto P_{x,j}(n) \hbox{ mod } 1$ is a polynomial of degree at most $m-2$.
			\item   For any $1 \leq i < j \leq m$, one has $\sup_{n \in \Z} P_{x,i}(n) + \sup_{n \in \Z} P_{x,j}(n) \leq 1$.
		\end{enumerate-math}
	\end{prop}
	
	We remark that the polynomiality property (ii) was previously observed in \cite[Lemma 4.3]{BH}.  The linear forms $\alpha m_1 \hbox{ mod } 1, \alpha m_2 \hbox{ mod } 1, \alpha (m_1+m_2) \hbox{ mod } 1$ implicitly appearing in Section \ref{counter-sec} are essentially examples of the polynomials appearing in Proposition \ref{pxj-prop}, in the context of higher level tilings.  However, we will eliminate this sort of ``equidistributed'' behavior in the level one case in Proposition \ref{coset} below.
	
	\begin{proof}[Proof of \propref{pxj-prop}]  If starts with \eqref{vph} and works modulo $1$ to eliminate the $\one_A$ and $\one_E$ terms, we have
		\begin{equation}\label{jsum}
		\sum_{j=1}^m \varphi_j = 0 \hbox{ mod } 1.
		\end{equation}
		Evaluating this at $x+ne$ we obtain (i).
		
		Now we prove (ii).  Let $1 \leq j \leq m$. From \eqref{nlat} we see that for any $1 \leq i \leq m$ distinct from $j$, that 
		$$e = qN\tilde e \in q\langle h_i, h_j \rangle,$$
		that is to say we have
		\begin{equation}\label{e}
		e = a_{i,j} qh_i + b_{i,j} qh_j
		\end{equation} 
		for some integers $a_{i,j}, b_{i,j}$.  In particular, from the $\langle qh_i\rangle$-periodicity of $\varphi_i$ we have
		$$ \Delta_{e-b_{i,j} qh_j} \varphi_i = 0$$
		while from the $\langle qh_j \rangle$-periodicity of $\varphi_j$ we have
		$$ \Delta_{e-b_{i,j} qh_j} \varphi_j = \Delta_e \varphi_j.$$
		If we then apply the discrete derivative operators $\Delta_{e-b_{i,j} qh_j}$ for each $1 \leq i \leq m$ distinct from $j$ in turn to \eqref{jsum} to eliminate all the $\varphi_i$ other than $\varphi_j$, we conclude that
		$$ \Delta_e^{m-1} * \varphi_j = 0 \hbox{ mod } 1.$$
		(Note that the discrete derivative operators only involve convolution with integer-valued functions and are well defined on functions that are only defined modulo $1$.)  Evaluating this on the line $\{ x + ne: n \in \Z\}$ using \eqref{Pdef}, we conclude the one-dimensional identity
		$$ \Delta_1^{m-1} P_{x,j} = 0 \hbox{ mod } 1.$$
		That is to say, the $(m-1)^{\mathrm{th}}$ discrete derivative of $P_{x,j} \hbox{ mod } 1$ vanishes.  A simple induction on $m$ then shows that $P_{x,j} \hbox{ mod } 1$ is a polynomial of degree at most $m-2$, giving (ii) as claimed.
		
		Now we prove (iii).  Let $1 \leq i < j \leq m$.
		As in \eqref{e}, we write $e$ as a linear combination \eqref{e} of $qh_i, qh_j$.  In particular, for any $n_1,n_2 \in \Z$ one has the identity
		$$ (x+n_1 e) + a_{i,j}(n_2-n_1)qh_i = (x+n_2 e) - b_{i,j} (n_2-n_1) qh_j.$$
		Evaluating \eqref{level-1} at this point, we conclude in particular that
		$$ \varphi_i( (x+n_1 e) + a_{i,j}(n_2-n_1)qh_i ) +  \varphi_j( (x+n_1 e) - b_{i,j}(n_2-n_1)qh_j ) \leq 1.$$
		Using \eqref{Pdef}, the $\langle q h_i \rangle$-periodicity of $\varphi_i$, and the $\langle qh_j \rangle$-periodicity of $\varphi_j$, we conclude that
		$$ P_{x,i}(n_1) + P_{x,j}(n_2) \leq 1.$$
		Taking suprema in $n_1,n_2$, we obtain (iii).
	\end{proof}

	To exploit these properties, we use the following exponential sum estimate\footnote{We thank Igor Shparlinski for this reference.} from analytic number theory.
	
	\begin{lem}[Exponential sum estimate]\label{lud}  Let $d \geq 1$, and let $P \colon \Z \to \R/\Z$ be a nonconstant polynomial of degree at most $d$ whose nonconstant coefficients are rational with denominators having least common multiple $Q$; in particular $P$ is periodic with period $Q$.  Then one has
		$$ \left| \sum_{n=1}^Q e^{2\pi i P(n)}\right| \ll_d  Q^{1-\frac{1}{d}}.$$
	\end{lem}
	
	\begin{proof}  See \cite{stechkin}.
	\end{proof}
	
	We also remark that bounds of this type (but with the gain $1/d$ replaced by an exponent that decays exponentially in $d$) can be established by the Weyl differencing (or van der Corput) method; see for instance \cite[Lemma 1.1.16]{tao}.
	
	\begin{lem}[Polynomials only fail to equidistribute when periodic]\label{period}  Let $P \colon \Z \to \R/\Z$ be a polynomial of some degree $d$ that avoids an interval $I$ of positive length in $\R/\Z$.  Then $P$ is periodic with period at most $O_{d,I}(1)$.
	\end{lem}
	
	\begin{proof}  Since $P$ is not equidistributed\footnote{We say that a function $P \colon \Z \to \R/\Z$ is equidistributed if we have $\lim_{N \to \infty} \frac{1}{2N+1} \sum_{n=-N}^N F(P(n)) = \int_{\R/\Z} F(t)\ dt$ for all continuous $F \colon \R/\Z \to \R$.}, we see from the Weyl equidistribution theorem that all the non-constant coefficients of $P$ must be rational.  We let $Q$ be the least common multiple of the denominators of these coefficients, then $P$ is periodic with period $Q$.  By the Weierstrass approximation theorem, one can find a trigonometric polynomial $f$ (depending only on $I$) that is periodic with period $1$, has mean zero, and is at least $1$ outside of $I$.  Then
		$$ \frac{1}{Q} \sum_{n=1}^Q f( P(n) ) \geq 1.$$
		By the pigeonhole principle, we can thus find a non-zero integer $k=O_I(1)$ such that
		$$ \left|\frac{1}{Q} \sum_{n=1}^Q e^{2\pi i kP(n)}\right| \gg_I 1.$$
		Note that the least common multiple of the denominators of $kP(n)$ is $\gg_{d,I} Q$.  Applying Lemma \ref{lud}, we conclude that $Q = O_{d,I}(1)$, giving the claim.
	\end{proof}
	
	Let $M$ be the least common multiple of $|\tilde e \wedge h_j|$ for $1 \leq j \leq m$.
	Let $\Lambda \coloneqq \langle e, qNMh_1\rangle$ be the lattice generated by $e, qNMh_1$.  Note in particular that
	$$ \Lambda \subset q \Z^2 \subset \ell \Z^2.$$
	
	\begin{prop}[One-dimensional periodicity in each coset]\label{coset}  On each coset $x + \Lambda$, the set $A \cap (x + \Lambda)$ is $\langle Q_m M^2 N q h_j \rangle$-periodic for some $1 \leq j \leq m$, where $Q_m$ depends only on $m$  and $Q_m M^2 N qh_j \in \Lambda$.
	\end{prop}
	
	\begin{proof}  
		From \eqref{cramer-include} one has
		$$ M \Z^2 \subset \langle \tilde e, h_i \rangle$$
		for all $1 \leq i \leq m$.  In particular, for each $1 \leq i \leq m$ one has
		\begin{equation}\label{qnm-decomp}
		qNMh_1 = a_i qh_i + b_i e
		\end{equation}
		where $e=qN\tilde e$ and $a_i,b_i$ are the integers
		\begin{equation}\label{aj-def}
		a_i \coloneqq NM \frac{h_1 \wedge \tilde e}{h_i \wedge \tilde e} 
		\end{equation}
		and
		$$ b_i \coloneqq M \frac{h_1 \wedge h_i}{\tilde e \wedge h_i}.$$ 
		In particular we have
		$$ \Lambda = \langle e, a_i qh_i \rangle$$
		for all $1 \leq i \leq m$.
		
		Fix $x \in \Z^2$.  From \eqref{qnm-decomp} we have for any integers $s,t$ that
		\begin{align*}
		\varphi_i( x + s e + t qNMh_1 ) &= \varphi_i( x + (s+b_i t) e + t a_i qh_i) \\
		&=  \varphi_i( x +(s+b_i t) e)\\
		&= P_{x,i}( s+ b_i t ) 
		\end{align*}
		thanks to the $\langle qh_i\rangle$-periodicity of $\varphi_i$, and \eqref{Pdef}.  Also, since $q$ is a multiple of $\ell$, $e$ is a multiple of $q$, and $E$ is $\ell \Z^2$-periodic, one has
		$$ \one_E( x + se + t qNMh_1 ) = \one_E(x).$$
		Thus we have the decomposition
		\begin{equation}\label{decomp}
		\one_A(x + s e + t qNMh_1 ) = \one_E(x) - \sum_{j=1}^m P_{x,j}( s + b_i t).
		\end{equation}
		
		Suppose first that $\sup_{n \in \Z} P_{x,j}(n) = 1$ for some $1 \leq j \leq m$.  Then from Proposition \ref{pxj-prop}(iii) we have $P_{x,i}(n)=0$ for all other $i$, hence by Proposition \ref{pxj-prop}(i), $P_{x,j}$ takes values in $\{0,1\}$.  The decomposition \eqref{decomp} then simplifies to
		$$ \one_A(x + s e + t qNMh_1 ) = \one_E(x) - P_{x,j}( s + b_j t)$$
		for all $s,t \in \Z$.  From \eqref{qnm-decomp} and the change of variables $s' \coloneqq s + b_j t$ this implies that
		$$ \one_A(x + s' e + t a_j qh_j ) = \one_E(x) - P_{x,j}( s' )$$
		for all $s',t \in \Z$.  In particular, on the coset $x + \Lambda$, the function $\one_A$ is $\langle a_j qh_j\rangle$-periodic, and hence $\langle M^2 N q h_j \rangle$-periodic by \eqref{aj-def} and the construction of $M$. Note that this argument also shows that $M^2 N qh_j \in \Lambda$.
		
		Now suppose we are in the opposite case that
		$$ \sup_{n \in \Z} P_{x,j}(n) < 1$$
		for all $1 \leq j \leq m$. From Proposition \ref{pxj-prop}(iii) this implies that
		$$ \sup_{n \in \Z} P_{x,i}(n) \leq 1/2$$
		for all $1 \leq i \leq m$ with at most one exception; thus, with at most one exception, $P_{x,i} \hbox{ mod } 1$ takes values in $[0,1/2] \hbox{ mod } 1$.  By Proposition \ref{pxj-prop}(ii) and Lemma \ref{period} this implies that with at most one exception, the $P_{x,i} \hbox{ mod } 1$ are periodic with period $O_m(1)$. Using Proposition \ref{pxj-prop}(i) and \lemref{period}, an exception cannot occur. Taking a common denominator, we conclude that there is a positive integer $Q_m$ depending only on $m$, such that the $P_{x,i} \hbox{ mod } 1$ for all $1 \leq i \leq m$  are $\langle Q_m\rangle$-periodic; since all the $P_{x,i}$ have supremum strictly less than $1$, we see that the $P_{x,i}$ are also $\langle Q_m \rangle$-periodic (note we no longer work modulo $1$).  From \eqref{decomp} we conclude that on the coset $x + \Lambda$, $A$ is $\langle Q_m M N q h_1 \rangle$-periodic.  Thus in either case we obtain the proposition.
	\end{proof}
	
	We finally conclude \thmref{main-quant}.
	\begin{proof}[Proof of \thmref{main-quant}]
	 From \eqref{h-prod} we can bound
	\begin{align*}
	M &\leq \prod_{j=1}^m \|h_j\| \|\tilde e\| \\
	&\ll_{|F|} \mathrm{diam}(F)^{|F|-1}
	\end{align*}
	and then from this and \eqref{h-prod}, \eqref{q-bound}, \eqref{N-bound} the lattice $\Lambda \coloneqq \langle qN\tilde e, qNMh_1 \rangle$ has index in $\ell\Z^2$  at most
	\begin{align*}
	qN\|\tilde e\| qNM\|h_1\| / \ell^2 &\ll_{|F|} N^2 M \mathrm{diam}(F)^{|F|-1} \\
	&\ll_{|F|} \mathrm{diam}(F)^{2(|F|-1)^2}.
	\end{align*}
	A similar computation shows that the quantity $L \coloneqq Q_m M^2 N q / \ell$ is an integer of magnitude at most
	\begin{align*}
	L &\ll_{|F|} M^2 N   \\
	&\ll_{|F|} \mathrm{diam}(F)^{(|F|-1)|F|}.
	\end{align*}
	This establishes Theorem \ref{main-quant} (note from an inspection of the arguments that none of the quantities $h_1,\dots,h_m, L, \Lambda$ constructed depends directly on $A$).
\end{proof}

	\section{Converting weakly periodic tilings to periodic tilings}\label{periodic-sec}
	
	A simple pigeonholing  argument\footnote{This argument is analogous to Newman's proof of the one-dimensional periodic tiling conjecture (\cite{N}).} shows that the existence of a weakly periodic tiling implies the existence of a periodic tiling (see, e.g., \cite[Theorem 2.3]{BH}, \cite[Proposition 2.1]{RX}). In this section we  introduce a constructive way to convert any weakly periodic tiling to a periodic tiling. This constructive approach allows us to achieve the significantly better bound  in \thmref{period-quant}, on the periods of the obtained periodic tiling. This bound is polynomial in the diameter of the tile $F$ (for fixed $|F|$).

	 We first give a ``slicing lemma'' that allows one to ``slice'' the tile $F$ along cosets along which the tiling set $A$ already exhibits some periodicity, while retaining the periodicity of the set being tiled.
	
	\begin{lem}[Slicing lemma]\label{slice}  Let $\ell, k \geq 1$ be natural numbers, and let $h$ be a primitive element of $\Z^2$.  Suppose that $F$ is a finite subset of $\Z^2$, and $A$ is a $\langle \ell k h\rangle$-periodic subset of $\Z^2$ such that $\one_F * \one_A$ is $\ell \Z^2$-periodic.  Then for every coset $x+\langle h \rangle$ of $\langle h \rangle$, the function $\one_{F \cap (x+\langle h \rangle)} * \one_A$ is $qks\Z^2$-periodic, where $q$ is the least common multiple of $\ell$ and all the primes less than or equal to $2|F|$, and $s$ is the least common multiple of the $|h \wedge (f-f')|$ for all $f,f' \in F$ with $f-f'$ incommensurate with $h$.
	\end{lem}
	
	\begin{proof} We may assume that $F$ intersects $x + \langle h \rangle$, as the claim is trivial otherwise; by relabeling we may then assume $x \in F$.  By translating $F$ and $x$ we then may assume that $x=0$ and $0 \in F$.  From Lemma \ref{dilation}(iii) we have
		$$ \one_F * \one_A = \one_{rF} * \one_A$$
		whenever $r = 1 \hbox{ mod } q$.  If we strengthen the condition on $r$ to $r = 1 \hbox{ mod } qk$, then for each $f \in F \cap \langle h \rangle$, we have $rf - f \in \langle qk h \rangle \subset \langle \ell k h \rangle$, hence by the $\langle \ell k h\rangle$-periodicity of $A$ we have
		$$ (\one_{r(F \cap \langle h\rangle)} - \one_{F \cap \langle h \rangle}) * \one_A = 0.$$
		Combining the two equations, we see that
		$$ \one_F * \one_A = \one_{F \cap \langle h \rangle} * \one_A + \sum_{f \in F \backslash \langle h \rangle} \delta_{rf} * \one_A$$
		and thus
		$$ \one_{F \cap \langle h \rangle} * \one_A = \one_F * \one_A - \sum_{f \in F \backslash \langle h \rangle} \delta_{rf} * \one_A$$
		whenever $r = 1 \hbox{ mod } qk$.  Note that all the terms on the right-hand side are $\langle \ell k h \rangle$-periodic, since $\one_A$ is.  Averaging over $r$ and using the Arzel\'a-Ascoli theorem to extract a limit as in the proof of Theorem \ref{structure}, we see that
		$$ \one_{F \cap \langle h \rangle} * \one_A = \one_F * \one_A - \sum_{f \in F \backslash \langle h \rangle} \varphi_f$$
		for some functions $\varphi_f$ which are both $\langle \ell k h \rangle$-periodic and $\langle qk f\rangle$-periodic. By \eqref{cramer-include} and the definition of $s$ and $q$, each $\varphi_f$ is $qks\Z^2$-periodic, as is $\one_F * \one_A$, and the claim follows.
	\end{proof}

    \begin{remark}
         The above argument is in fact applicable to any $d$-dimensional one-periodic tiling of a periodic set. More precisely, it shows that if $F\subset\Z^d$ is finite, $h\in\Z^d$ is primitive and $A\subset\Z^d$ is a $\langle kh\rangle$-periodic set such that $\one_F*\one_A$ is periodic, then   the ``slice'' $\one_{F \cap \langle h \rangle}*\one_A$ can be written as $$\one_{F \cap \langle h \rangle}*\one_A=\one_F*\one_A-\sum_{f\in F\setminus \langle h\rangle}\varphi_{f},$$ where for any $f\in F\setminus \langle h\rangle$, the function $\varphi_f:\Z^d\to[0,1]$ is $\langle  kh,qkf\rangle$-periodic, for some $q(F)\in \Z$.
    \end{remark}
	
	We now use the slicing lemma to improve the periodicity properties of a tiling. The slicing lemma allows us to reduce the two-dimensional tiling to a collection of one-dimensional tilings, by considering slices (or ``cut and project'' sets) of both the tile $F$ and the tiling set $A$.
	We remark that similar argument was used in \cite{JP} and \cite{GL1} in a different context. 
	
	\begin{corollary}[Improving a tiling, I]\label{improv}  Let $\ell, k \geq 1$ be natural numbers, and let $h$ be a primitive element of $\Z^2$.  Suppose that $F$ is a finite subset of $\Z^2$, and $A$ is a $\langle \ell k h\rangle$-periodic subset of $\Z^2$ such that $\one_F * \one_A$ is $\ell \Z^2$-periodic.  Then there exists a $qks\Z^2$-periodic set $A'$ such that $\one_F * \one_A = \one_F * \one_{A'}$, where $q,s$ are as in Lemma \ref{slice}.  In fact we have the stronger statement
	$$ \one_F * \one_{A \cap (\langle h \rangle + y)} =  \one_F * \one_{A' \cap (\langle h \rangle + y)}$$
	for any coset $\langle h \rangle+y$ of $\langle h \rangle$.
	\end{corollary}
	
	\begin{proof}  By applying an invertible linear transformation in $\mathrm{SL}_2(\Z)$ we may assume without loss of generality that the primitive element $h$ is equal to $(1,0)$.  For every integer $y \in\Z$, we introduce the one-dimensional slices
		$$ F_y \coloneqq \{ x \in \Z: (x,y) \in F \}$$
		and
		$$ A_y \coloneqq \{ x \in \Z: (x,y) \in A \}.$$
		From Lemma \ref{slice} we know that for each $y$, the set $\one_{F_y \times \{y\}} * \one_A$ is $qks\Z^2$-periodic, which in particular implies on taking slices that
		$$ \one_{F_y} * \one_{A_z} =  \one_{F_y} * \one_{A_{z+qks}}$$
		for all $y,z \in \Z$, that is to say the map $z \mapsto \one_{F_y} * \one_{A_z}$ is periodic in $z$.
		
		Let $\Sigma$ denote the collection of all $\langle \ell k\rangle$-periodic subsets of $\Z$; this is a finite set that contains $A_z$ for all $z \in \Z$.  Introduce an equivalence relation on $\Sigma$ by declaring $B \sim B'$ if $\one_{F_y} * \one_{B} = \one_{F_y} * \one_{B'}$ for all $y \in \Z$.  Thus we have $A_{z+qks} \sim A_z$ for all $z \in \Z$.  Now arbitrarily place a total ordering on the finite set $\Sigma$, and for each $z \in \Z$ let $\tilde A_z \in \Sigma$ denote the minimal element in the equivalence class $\{ B \in \Sigma : B \sim A_z\}$.  Then we have $\tilde A_{z+qks} = \tilde A_z \sim A_z$ for all $z \in \Z$.  If we then define the modified set
		$$ \tilde A \coloneqq \{ (x,z): z \in \Z, x \in \tilde A_z \}$$
		then $\tilde A$ is $qks\Z^2$-periodic, and we have
		$$ \one_{F_y \times \{y\}} * \one_{\tilde A} = \one_{F_y \times \{y\}} * \one_A$$
		for all $y$, which on taking horizontal slices implies that
		$$ \one_{F_y \times \{y\}} * \one_{\tilde A \cap (\Z \times \{z\})} = \one_{F_y \times \{y\}} * \one_{A \cap (\Z \times \{z\})}$$
		for all $y,z \in \Z$.  Summing in $y$, we conclude that
		$$ \one_{F} * \one_{\tilde A \cap (\Z \times \{z\})} = \one_{F} * \one_{A \cap (\Z \times \{z\})},$$
		giving the claim.
	\end{proof}

	Now we combine this corollary with the argument used to establish Proposition \ref{pxj-prop}(ii) to convert weakly periodic tilings to periodic tilings.
	
	\begin{thm}[Improving a tiling, II]\label{improv-ii}  Let $\ell \geq 1$ and $m \geq 2$ be natural numbers, and let $h_1,\dots,h_m$ be pairwise incommensurable elements of $\Z^2$.  Suppose that $F$ is a finite subset of $\Z^2$, that $E$ is an $\ell\Z^2$-periodic subset of $\Z^2$, and $A$ is a tiling of $E$ by $F$.  Suppose that $A$ is weakly periodic, and more specifically that $A$ is the disjoint union $A = A_1 \uplus \dots \uplus A_m$ where each $A_j$ is $\langle \ell h_j\rangle$-periodic.  Then there exists a $\ell M \Z^2$-periodic set $A'$ such that $\one_F * \one_A = \one_F * \one_{A'} = \one_E$, where $M$ is an integer with the bound
		$$ M \ll_{|F|} (\prod_{i=1}^m \|h_i\|)^{m + |F|(|F|-1)/2} \mathrm{diam}(F)^{m|F|(|F|-1)/2}.$$
		Furthermore each $\one_F * \one_{A_j}$ is $\ell N_j \Z^2$-periodic for some integer $N_j$ with bound
		$$ N_j \ll_{|F|} (\prod_{i=1}^m \|h_i\|)^{m-1} \|h_j\|.$$
	\end{thm}
	
	\begin{proof}  We can locate an element $\tilde e \in \Z^2$ of magnitude $O_m(1)$ which is incommensurable with all of the $h_i$.  We define $N$ to be the least common multiple of the $|h_i \wedge h_j|$ for $1 \leq i < j \leq m$, and set $e \coloneqq \ell N \tilde e$.  Observe that
		$$ N \leq \prod_{1 \leq i < j \leq m} |h_i \wedge h_j| \leq (\prod_{i=1}^m \|h_i\|)^{m-1}.$$
		
		Let $1 \leq j \leq m$.  We have the identity
		$$ \one_F * \one_{A_j} + \sum_{i\in\{1,\dots,m\}\setminus\{j\}} \one_F * \one_{A_i} = \one_E$$
		To eliminate all the terms besides $ \one_F * \one_{A_j}$ we apply discrete differentiation operators as in the proof of Proposition \ref{pxj-prop}(ii).  From \eqref{cramer-include} we can write
		$$ e = a_{i,j} \ell h_i + b_{i,j} \ell h_j$$
		for all $1 \leq i \leq m$ distinct from $j$ and some integers $a_{i,j}, b_{i,j}$.  Then from the periodicity properties of the $A_i$, $A_j$ we have
		$$ \Delta_{e - b_{i,j} \ell h_j} (\one_F * \one_{A_i}) = 0$$
		and
		$$ \Delta_{e - b_{i,j} \ell h_j} (\one_F * \one_{A_j}) = \Delta_e (\one_F * \one_{A_j});$$
		also from the $\ell\Z^2$-periodicity of $E$ one has
		$$ \Delta_{e - b_{i,j} \ell h_j} \one_E = 0.$$
		Applying each of the discrete derivative operators $\Delta_{e - b_{i,j} \ell h_j}$ in turn, we conclude that
		$$ \Delta_e^{m-1} (\one_F * \one_{A_j}) = 0.$$
		Thus, for any $x$, the map $n \mapsto \one_F * \one_{A_j}(x+ne)$ is polynomial in $n$; but it is also bounded, hence constant.  Thus $\one_F * \one_{A_j}$ is $\langle e \rangle$-periodic and $\langle \ell h_j\rangle $-periodic, hence by \eqref{cramer-include} it is $\ell N  M_j \Z^2$-periodic for some integer $M_j$ with $M_j \ll_m \|h_j\|$. We now split $h_j = k_j h'_j$ where $k_j$ is a natural number and $h'_j$ is primitive.  By Corollary \ref{improv} we see that we can find a $q_j k_j s_j \Z^2$-periodic set $A'_j$ such that $\one_F * \one_{A_j} = \one_F * \one_{A'_j}$, where $q_j = \ell N M_j C_j$ for some $C_j = O_{|F|}(1)$, and $s_j$ is a positive integer with the bound
		$$ s_j \leq (\|h'_j\| \mathrm{diam}(F))^{|F|(|F|-1)/2}.$$
		We conclude that $q_j k_j s_j = \ell N L_j$ with
		$$ L_j \ll_{|F|} \|h_j\|^{1+|F|(|F|-1)/2} \mathrm{diam}(F)^{|F|(|F|-1)/2}.$$
		Summing over $j$, we have
		$$ \one_F * \sum_{j=1}^m \one_{A'_j} = \one_F * \one_A = \one_E.$$
		In particular this shows that $\sum_{j=1}^m \one_{A'_j}$ is bounded by $1$, that is to say the $A'_j$ are disjoint.  If we  set $A' \coloneqq \bigcup_{j=1}^m A'_j$, we then have $\one_F * \one_A = \one_F * \one_{A'} = \one_E$, and $A'$ is $\ell M \Z^2$-periodic with
		$$ M \leq N \prod_{j=1}^m L_j \ll_{|F|} (\prod_{i=1}^m \|h_i\|)^{m + |F|(|F|-1)/2} \mathrm{diam}(F)^{m|F|(|F|-1)/2}$$
		giving the claim.
	\end{proof}

   We can now conclude the proof of \thmref{period-quant}.
    
	\begin{proof}[Proof of \thmref{period-quant}] 
		The claim is trivial for $|F|=1$, so suppose $|F|>1$.  By Theorem \ref{main-quant}, there exist $m,h_1,\dots,h_m,L$ obeying the conclusions of that theorem, such that $A$ is the disjoint union of $\langle \ell L h_j \rangle$-periodic sets for $j=1,\dots,m$.  Applying Theorem \ref{improv-ii} (with $\ell$ replaced by $\ell L$), we can find an $\ell L M' \Z^2$-periodic set $A'$ with $\one_F * \one_A = \one_F * \one_{A'} = \one_E$, where
	$$ M' \ll_{|F|} (\prod_{i=1}^m \|h_i\|)^{\tfrac{1}{2}(|F|+2)(|F|-1)} \mathrm{diam}(F)^{\tfrac{1}{2}|F|(|F|-1)^2}.$$
	Using the bounds \eqref{h-b}, \eqref{l-b}, we have
	$$ LM' \ll_{|F|}\mathrm{diam}(F)^{|F|(|F|-1)} \mathrm{diam}(F)^{\tfrac{1}{2}(|F|+2)(|F|-1)^3}   \mathrm{diam}(F)^{\tfrac{1}{2}|F|(|F|-1)^2}.$$
	The claim follows.
	\end{proof}
	
	\section{One-periodic tilings}\label{oneper-sec}
	
	Call a set $A$ \emph{one-periodic} if $A$ is $\langle h \rangle$-periodic for some non-zero $h$.  For any given tile $F$ of $\Z^2$, we consider\footnote{We thank Mihalis Kolountzakis (private communication) for suggesting this question.} the following question:
	
	\begin{question}\label{one-per}  Let $F$ be a finite tile of $\Z^2$.
		Does there exist a tiling of $\Z^2$ by $F$ which is not one-periodic?
	\end{question}
	
	The answer to this question is positive for some tiles and negative for others, as the following examples show:
	
	\begin{enumerate-math}
	\item If $F_0$ is the square $\{0,1\}^2$, then every tiling of $\Z^2$ by $F_0$ can be seen to be either $\langle (2,0) \rangle$-periodic or $\langle (0,2)\rangle$-periodic, and hence the answer to \quesref{one-per} is negative in this case. More generally, from \cite[Theorem 19]{szegedy} it follows that the answer to \quesref{one-per} is negative for any tile of cardinality $4$ that contains zero and generates $\Z^2$.
	\item On the other hand, in \subsecref{sub-periodicity} it was shown that the four-element tile $F_1 \coloneqq \{0,2\}\times\{0,1\}$ (which does \emph{not} generate $\Z^2$) admits tilings $A$ of $\Z^2$ that are not one-periodic.  Note that if $A$ is a tiling of $\Z^2$ by $F_2$, then $2A$ is a tiling of $\Z^2$ by $F_2 \coloneqq 2F_1 + \{0,1\}^2 =  \{0,4\}\times\{0,2\}+\{0,1\}^2$.  Thus the answer to \quesref{one-per} is positive for the tiles $F_1, F_2$.  In particular it is possible to have non-one-periodic tilings even when the tile $F$ contains $0$ and generates all of $\Z^2$ as an abelian group.
	\item If the tile $F$ is collinear (it lies in a one-dimensional affine subspace of $\R^2$), then Corollary \ref{structure-1d} implies that all tilings of $\Z^2$ by $F$ are $\langle h \rangle$-periodic for some universal period $h$ that is parallel to the line that $F$ lies in.  Hence the answer to Question \ref{one-per} is negative in this case.
	\item If the tile $F$ is ``connected'' in the sense that $F+[-\tfrac{1}{2},\tfrac{1}{2}]^2$ is simply connected, then  by \cite[Theorem 5.5]{gn}, $F$ admits only one-periodic tilings of $\Z^2$.  Hence the answer to \quesref{one-per} is negative in this case. This of course contains the first part of (i) as a special case.
	\item  From \cite[Theorem 17]{szegedy} it follows that if $F$ has prime order $|F|=p$ then every tiling of $\Z^2$ by $F$ is $p\Z^2$-periodic, so the answer to \quesref{one-per} is negative in this case.
	\end{enumerate-math}
	
	We do not have a full classification of the tiles $F$ which have a negative answer to \quesref{one-per}.  However, the results presented in our paper do at least show that this question is decidable for any given tile, even if one replaces the set $\Z^2$ by more general periodic subsets of $\Z^2$.
	
	\begin{thm}[Decidability of \quesref{one-per}]\label{algorithm}
		There is an algorithm which, when given periodic subset $E$ of $\Z^2$ and a finite set $F$ that tiles $E$ as input\footnote{Note that a periodic subset $E$ of $\Z^2$ can be stored using a finite amount of memory, by first storing a pair of generators for the periodicity lattice, and then storing a coset representative for each of the cosets of that lattice that lie in $E$.}, decides in finite time whether $F$ admits a non one-periodic tiling of $E$, or not.
	\end{thm}
	
	\begin{proof}	
		Suppose that $F$ is a tile of an $\ell\Z^2$ periodic set $E$.  From \thmref{main-quant} and \thmref{improv-ii} we can compute positive integers $m, N_1,\dots,N_m$ and pairwise incommensurable $h_1,\dots,h_m \in \Z^2$, with the property that for every tiling set $A$ of $E$ by $F$, there exists partitions $A = A_1 \uplus \dots \uplus A_m$ and $E = E_1 \uplus \dots \uplus E_m$, where for each $1 \leq j \leq m$, $A_j$ is $\langle h_j \rangle$-periodic, $E_j$ is $N_j \Z^2$-periodic, and $A_j$ is a tiling set for $E_j$ by $F$: $\one_F * \one_{A_j} = \one_{E_j}$.  We write $h_j = k_j h'_j$ for a (computable) natural number $k_j$ and primitive $h'_j$.  By Corollary \ref{improv}, we can then compute positive integers $M_1,\dots,M_m$ such that whenever $A,A_1,\dots,A_m,E_1,\dots,E_m$ are as above, one can find an $M_j \Z^2$-periodic set $A'_j$ for each $j=1,\dots,m$ such that
		$$ A_j \cap (\langle h'_j \rangle+y) \equiv_F A'_j \cap (\langle h'_j \rangle+y)$$
		for all cosets $\langle h'_j \rangle+y$, where we use $A \equiv_F A'$ to denote the equivalence relation $\one_F * \one_A = \one_F * \one_{A'}$.  By construction we have
		$$ \sum_{j=1}^m \one_F * \one_{A'_j} =  \sum_{j=1}^m \one_F * \one_{A_j} = \sum_{j=1}^m \one_{E_j} = \one_E.$$
		
		Conversely, if for each $j=1,\dots,m$ we locate a $M_j \Z^2$-periodic set $A'_j$ such that
		\begin{equation}\label{apj-eq}
		 \sum_{j=1}^m \one_F * \one_{A'_j} = \one_E
		 \end{equation}
		and then for each coset $A'_j \cap (\langle h'_j \rangle+y)$ we find a $\langle h_j \rangle$-periodic subset $A_{j,\langle h'_j \rangle+y}$ of $\langle h'_j \rangle+y$ such that 
		\begin{equation}\label{aj-apj}
		A_{j,\langle h'_j \rangle+y} \equiv_F A'_j\cap (\langle h'_j \rangle+y)
		\end{equation}
		then the sets
		$$ A_j \coloneqq \bigcup_{\langle h'_j \rangle+y \in \Z^2 / \langle h'_j \rangle} A_{j,\langle h'_j \rangle+y}$$
		are such that
		$$ A_j \equiv_F A'_j$$
		and hence
		$$ \sum_{j=1}^m \one_F * \one_{A_j} = \one_E.$$
		In particular, the $A_j$ are disjoint and their union 
		\begin{equation}\label{A-def}
		A \coloneqq \bigcup_{j=1}^m A_j= \bigcup_{j=1}^m \bigcup_{y\in\langle h'_j \rangle^\perp} A_{j,\langle h'_j \rangle+y}
		\end{equation}
		tiles $E$ by $F$.  Thus, we have a way of completely describing all the tilings $A$ of $E$ by $F$: we first locate all tuples $(A'_1,\dots,A'_m)$ of $M_j \Z^2$-periodic sets $A'_j$ obeying \eqref{apj-eq}, then for each such tuple, each $j=1,\dots,m$, and each coset $\langle h'_j \rangle+y$, we choose a $\langle h_j \rangle$-periodic subset $A_{j,\langle h'_j \rangle+y}$ obeying \eqref{aj-apj}, then the unions \eqref{A-def} give all the possible tilings of $E$ by $F$.
		
		Note that there are only finitely many $M_j\Z^2$-periodic subsets of $\Z^2$ (in fact there are $2^{M_j^2}$ many), so there are only finitely many possible tuples $(A'_1,\dots,A'_m)$ that can arise in this fashion.  For each such tuple, the identity \eqref{apj-eq} can be verified or disproved in finite time.  Hence we can completely enumerate all the possible tuples $(A'_1,\dots,A'_m)$ that can arise in the above construction.  Similarly, for each coset $\langle h'_j \rangle+y$ and any fixed choice of $A'_j$, the set of all $\langle h_j \rangle$-periodic subsets $A_{j, \langle h'_j \rangle+y}$ of $\langle h'_j \rangle+y$ obeying \eqref{aj-apj} is also finite and can be enumerated in finite time.  However, there are infinitely many such cosets, and hence one may potentially need to loop over an infinite number of choices in order to enumerate all possible tilings (cf. the tilings \eqref{a1-ex}, \eqref{a2-ex}).
		
	    Fortunately, for the question of whether there admit non-one-periodic tilings we do not need to exhaust over infinitely many possibilities, and can instead reason as follows.  Let us call a slice $A'_j\cap (\langle h'_j \rangle+y)$ of an $M_j \Z^2$-periodic set $A'_j$ by a coset $(\langle h'_j \rangle+y)$ a \emph{chameleon} if there exists a $\langle h_j \rangle$-periodic subsets $A_{j, \langle h'_j \rangle+y}$ of $\langle h'_j \rangle+y$ obeying \eqref{aj-apj} which is not equal to $A'_j\cap (\langle h'_j \rangle+y)$, in which case we say that $A'_j$ \emph{contains a chameleon slice}.  Note that for any given $M_j \Z^2$-periodic set $A'_j$, the slices of $A'_j\cap (\langle h'_j \rangle+y)$ repeat periodically in $y$ and hence one can determine whether $A'_j$ contains a chameleon slice or not in finite time.
	    
	    Suppose first that $A$ arises from a tuple $(A'_1,\dots,A'_m)$ with the property that at most one of the $A'_j$ admits a chameleon slice, thus we have some $j_0$ such that $A'_j$ does not admit a chameleon slice for any $j \neq j_0$.  Then for $j \neq j_0$, the set $A_{j, \langle h'_j \rangle+y}$ is necessarily equal to $A'_j\cap (\langle h'_j \rangle+y)$ for every $y$, hence $A_j$ is necessarily equal to the $M_j \Z^2$-periodic set $A'_j$. Since $A_{j_0}$ is also $\langle h_{j_0} \rangle$-periodic, we conclude that the set $A$ constructed in \eqref{A-def} is one-periodic in this case.
	    
	    Now suppose that $A$ arises from a tuple $(A'_1,\dots,A'_m)$ with the property that there are at least two sets $A'_{j_1}, A'_{j_2}$ which admit chameleon slices $A'_{j_1} \cap (\langle h'_{j_1}+y_1)$, $A'_{j_2} \cap (\langle h'_{j_2}+y_2)$ respectively.  Thus for $i=1,2$, we may find a $\langle h_{j_i}\rangle$-periodic subset $A^{(i)}$ of $\langle h'_{j_i} \rangle+y_i$ that is distinct from $(A')^{(i)} \coloneqq A'_{j_i} \cap \langle h'_{j_i} \rangle+y_i$ such that
	    $$ A^{(i)} \equiv_F (A')^{(i)}.$$
	    In particular, if we set $A' \coloneqq \bigcup_{j=1}^m A'_j$, and let $A$ be the set formed from $A'$ by replacing $(A')^{(i)}$ with $A^{(i)}$ for $i=1,2$, so that
	    $$ \one_A = \one_{A'} + \sum_{i=1}^2 (\one_{A^{(i)}} - \one_{(A')^{(i)}}),$$
	    then $A$ is equivalent to $A'$:
	    $$ \one_A * \one_F = \one_{A'} * \one_F = \one_E.$$
	    We claim that $A$ is not one-periodic.  Since $\one_{A'}$ is $M\Z^2$-periodic for some $M$, it suffices to show that $\sum_{i=1}^2 (\one_{A^{(i)}} - \one_{(A')^{(i)}})$ is not one-periodic.  But each summand $\one_{A^{(i)}} - \one_{(A')^{(i)}}$ is a non-trivial one-periodic function supported on the coset $\langle h'_{j_i} + y_i \rangle$, and the claim is then evident from the incommensurability of $h'_{j_1}$ and $h'_{j_2}$.
	    
	    From the above discussion, we conclude that $F$ admits non-one-periodic tilings if and only if there exist $M_j\Z^2$-periodic sets $A'_j$ for $j=1,\dots,m$ obeying \eqref{apj-eq} such that at least two of the $A'_j$ admit chameleon slices.  Since we can computably enumerate all the tuples $(A'_1,\dots,A'_m)$ obeying \eqref{apj-eq} and test all of them for chameleon slices, the claim follows.
	\end{proof}
	
	Informally, the proof of Theorem \ref{algorithm} tells us that all the tilings of a given periodic set $E$ by a given tile $F$ arise from starting with a doubly periodic tiling $A' = A'_1 \uplus \dots \uplus A'_m$ (which is drawn from a finite list of such tilings) and performing a (possibly infinite) number of ``slide moves'' in which one or more slices $A'_j\cap (\langle h'_j \rangle+y)$ is replaced with an equivalent (one-periodic) set $A_{j, \langle h'_j \rangle+y}$, in the spirit of \eqref{a1-ex} or \eqref{a2-ex}.  In principle, this should reduce any reasonable question about such tilings to a finite computation for any fixed choice of $E,F$, with the existence of non-one-periodic tilings serving as just one representative example of such a question.




\appendix 

\section*{Acknowledgments} 
We thank Jarkko Kari for alerting us to the relevant references \cite{k-survey}, \cite{ks}, \cite{szabados}. 
	We also thank Mihalis Kolountzakis for helpful suggestions, and Nishant Chandgotia and several readers of the second author's blog for further  comments and corrections.  We have also implemented several suggestions of the anonymous referee.

\bibliographystyle{amsplain}


\begin{dajauthors}
\begin{authorinfo}[RG]
  Rachel Greenfeld\\
  UCLA Department of Mathematics\\ 
  Los Angeles, CA 90095-1555\\
  
  greenfeld\imageat{}math.ucla\imagedot{}edu \\
  \url{https://www.math.ucla.edu/~greenfeld}
\end{authorinfo}
\begin{authorinfo}[tao]
  Terence Tao\\
  UCLA Department of Mathematics\\ 
  Los Angeles, CA 90095-1555\\
  tao\imageat{}math.ucla\imagedot{}edu \\ 
  \url{https://www.math.ucla.edu/~tao}
\end{authorinfo}

\end{dajauthors}

\end{document}